\documentclass[12pt, reqno, twoside, letterpaper]{amsart}

\usepackage{paperstyle}
\usepackage{graphicx}

\usepackage{mathtools}
\usepackage{todonotes}
\usepackage[norefs,nocites]{refcheck}
\usepackage{amsmath,amsthm,amscd,amssymb}

\newcommand{\be}{\begin{equation}}
\newcommand{\ee}{\end{equation}}
\newcommand{\dalign}[1]{\[\begin{aligned} #1 \end{aligned}\]}

\newcommand{\euD}{\EuScript{D}}

\newcommand{\euR}{\EuScript{R}}

\newcommand{\bfT}{{\bf T}}

\title[Shifting ordinates of zeros of the zeta function]
{Shifting the ordinates of\break zeros of the Riemann zeta function}
      
\author[W.~Banks]{William Banks}

\address{Department of Mathematics, 
         University of Missouri, 
         Columbia MO 65211, USA.}

\email{bankswd@missouri.edu}
        
\date{\today}

\begin{document}

\begin{abstract}
Let $y\ne 0$ and $C>0$. Under the Riemann Hypothesis,
there is a number $T_*>0$ $($depending on $y$ and $C)$
such that for every $T\ge T_*$, both
\[
\zeta(\tfrac12+i\gamma)=0
\mand\zeta(\tfrac12+i(\gamma+y))\ne 0
\]
hold for at least one $\gamma$ in the interval $[T,T(1+\eps)]$, where
$\eps\defeq T^{-C/\log\log T}$.
\end{abstract}

\thanks{MSC Primary: 11M26; Secondary 11M06}

\thanks{Keywords: Riemann zeta function, ordinates, zeros.}

\maketitle

\tableofcontents


\newpage

\vskip3cm

{\Large\section{Introduction}}

\subsection{Background and motivation}
The \emph{Riemann zeta function} is a central
object of study in number theory. The zeta function is defined in the
half-plane $\{s\in\C:\sigma\defeq\Re(s)>1\}$ by the equivalent formulas
$$
\zeta(s)\defeq\sum_{n\ge 1} n^{-s}
=\prod_{p\text{~prime}}(1-p^{-s})^{-1}.
$$
Riemann~\cite{Riemann} showed
that $\zeta(s)$ continues analytically to a meromorphic
function in the whole complex plane (the only singularity is a
simple pole at $s=1$), and it
satisfies a functional equation relating its values at
$s$ and $1-s$.\footnote{There are many excellent accounts of the
theory of the Riemann zeta function; we refer the reader to 
Titchmarsh~\cite{Titchmarsh} and to Borwein \emph{et al}~\cite{Borwein}
for essential background.}

Because the distribution of prime numbers is determined by the
location of the nontrivial zeros of $\zeta(s)$, the study of such zeros
is of paramount importance in prime number theory. The 
\emph{Riemann Hypothesis} (RH) asserts that
if $\rho=\beta+i\gamma$ is a zero of $\zeta(s)$ such that $\beta>0$,
then $\beta=\frac12$. If RH is true, then one obtains as a consequence
an optimal estimate for the counting function $\pi(x)$
of the primes $p\le x$. However, despite extraordinary advances towards RH by
Hardy, Selberg, Levinson, Conrey, and others, the Riemann Hypothesis
still remains far from settled.
In the words of Selberg, ``...there have been very few attempts at proving the Riemann hypothesis, because, simply, no one has ever had any really good idea for how to go about it.''

The Riemann Hypothesis provides a complete description of the horizontal 
distribution of the zeros of $\zeta(s)$. However, even assuming
that RH is true, very little is known unconditionally about the vertical distribution of the
zeros on the critical line $\{\sigma=\frac12\}$.
In a pioneering paper of great importance,
Montgomery~\cite{PairCorr1} proposed his celebrated
\emph{Pair Correlation Conjecture} (PCC), which (under RH) asserts
that the pair correlation between pairs of zeros of $\zeta(s)$
(normalized to have unit average spacing) is
$1-((\sin\pi u)/\pi u)^2$.
As Dyson had pointed out to Montgomery, this is the same as the pair
correlation function of random Hermitian matrices. More precisely, 
PCC asserts that for any $\alpha\le\beta$, one has
\[
\lim\limits_{T\to\infty}\frac{\big|\{(\gamma,\gamma'):0<\gamma,\gamma'\le T,
~\frac{2\pi\alpha}{\log T}\le\gamma'-\gamma
\le \frac{2\pi\beta}{\log T}\}\big|}
{\frac{1}{2\pi}T\log T}
=\int_\alpha^\beta\bigg\{1-\Big(\frac{\sin\pi u}{\pi u}\Big)^{\!2}\bigg\}\,du.
\]
Although PCC has significantly enhanced our conjectural understanding of
the gaps $\gamma'-\gamma$ between ordinates $\gamma,\gamma'$ 
of different zeros of $\zeta(s)$, and it has led to other conceptual
advances, it says nothing about how often gaps of a particular size $y\ne 0$
might occur. Indeed, this question appears to have been overlooked in the 
literature thus far. The aim of this paper, therefore, is to establish
some nontrivial results on this question. Our methods are flexible, and
the machinery we use can be adapted to treat other similar questions.

\subsection{Statement of results}
As the gaps between zeros of $\zeta(s)$
on the critical line form a \emph{countably infinite} set, for any specific
$y\ne 0$, it is natural to predict that the relation $\gamma'-\gamma=y$ holds for \emph{at most finitely many} pairs $(\gamma,\gamma')$; in other words,
one expects that $\zeta(\tfrac12+i(\gamma+y))\ne 0$ almost always
when $\zeta(\tfrac12+i\gamma)=0$.

\begin{theorem}\label{thm:main}
Assume RH. For any fixed $y\ne 0$ and $C>0$, there is a number
$T_*=T_*(y,C)>0$ such that for every $T\ge T_*$ we have
\[
\zeta(\tfrac12+i\gamma)=0
\mand\zeta(\tfrac12+i(\gamma+y))\ne 0
\]
for at least one $\gamma$ in the interval $[T,T(1+\eps)]$, where
$\eps\defeq T^{-C/\log\log T}$.
\end{theorem}

To prove the theorem, we study sums of the form
\[
\ssum{\rho=\frac12+i\gamma\\T_1<\gamma<T_2}x^{\rho}\zeta(\rho+iy),
\]
where $T_1$ is large and $T_1\asymp T_2$.
To estimate such sums, we apply the residue theorem with
the meromorphic function
\[
\euD_y(s)\defeq\frac{\zeta'}{\zeta}(s)\zeta(s+iy).
\]
Note that $\euD_y(s)$ can be represented
in the half-plane $\{\sigma>1\}$ by
an absolutely convergent Dirichlet series $\sum_n D_y(n)n^{-s}$ 
whose coefficients are given by
\[
D_y(n)\defeq -\sum_{ab=n}\Lambda(a)b^{-iy}\qquad(n\in\N),
\]
where $\Lambda$ is the von Mangoldt function.
For later reference, we note that
\be\label{eq:Dyn lower bound}
D_y(p)=-\log p\qquad(p\text{~prime}).
\ee

\begin{theorem}\label{thm:main2}
Assume RH. For any $y\ne 0$ and $A>0$, there are numbers
$T_0=T_0(y,A)>0$ and $\Theta=\Theta(A)>0$ for which the following holds.
For any $T_1,T_2$ satisfying
$2T_1>T_2>T_1>T_0$, let
\[
\Delta\defeq T_2-T_1,\qquad
\bfT\defeq\tfrac12(T_1+T_2),\qquad
\sL\defeq\exp\bigg(\frac{\log\bfT}{\log\log\bfT}\bigg).
\]
Then, for every prime number $x$ in the range
\be\label{eq:prime range}
\bfT\sL^{-\Theta}<x<\er^{\pi/|y|}\cdot\bfT\sL^{-\Theta},
\ee
we have
\be\label{eq:super}
\ssum{\rho=\frac12+i\gamma\\T_1<\gamma<T_2}x^\rho\zeta(\rho+iy)
=\tfrac{1}{2\pi}(x^{-iy}-1)\Delta\log\bfT
+O\bigg(\frac{\Delta\log\bfT}{(\log\log\bfT)^{1/2}}+\bfT\sL^{-A}\bigg),
\ee
where the sum in \eqref{eq:super} runs over zeros of $\zeta(s)$
$($each zero is summed according to its multiplicity$\,)$,
and the implied constant in \eqref{eq:super} depends only on $y$ and $A$.
\end{theorem}

{\Large\section{Auxiliary results}}

Throughout this section, any implied constants in the symbols
$O$, $\ll$, and $\gg$ are \emph{absolute} unless specified otherwise.
We denote $\e(u)\defeq\er^{2\pi iu}$ for all $u\in\R$.

\bigskip\subsection{Proof of Theorem~\ref{thm:main}}

In this subsection, we show that Theorem~\ref{thm:main} is a
 consequence of Theorem~\ref{thm:main2} coupled
with the following observation.

\begin{lemma}\label{lem:prime find}
Let $y\ne 0$. For every sufficiently large $t$, depending only on $y$,
there is a prime $p$ in the range $t<p<\er^{\pi/|y|}t$
such that $|p^{-iy}-1|>\frac{1}{\sqrt{2}}$.
\end{lemma}

\begin{proof}
Let $\alpha\defeq\er^{\pi/(4|y|)}$, which exceeds one.
Put $\eps\defeq |y|/y\in\{\pm 1\}$. For any 
$\beta\in(\alpha^2,\alpha^4)$, the number $\beta^{-iy}$
lies on the unit circle in $\C$ between $(\alpha^2)^{-iy}=(-i)^\eps$
and $(\alpha^4)^{-iy}=-1$; therefore,
$|\beta^{-iy}-1|>\sqrt{2}$.

By the Prime Number Theorem, for any  large
$t$ (depending on $y$), there exist primes
$p_1\in(t,\alpha t)$ and $p_2\in(\alpha^3t,\alpha^4t)$. The
number $\beta\defeq p_2/p_1$ lies in $(\alpha^2,\alpha^4)$,
and so by the preceding argument we have
\[
\big|p_2^{-iy}-p_1^{-iy}\big|=
\big|p_1^{-iy}(\beta^{-iy}-1)\big|>\sqrt{2}.
\]
This implies that $|p_j^{-iy}-1|>\frac{1}{\sqrt{2}}$
for at least one of the primes $p_j$.
\end{proof}

\begin{proof}[Proof of Theorem~\ref{thm:main}]

We apply Theorem~\ref{thm:main2} with the choices
$A\defeq 2C$, $T_1\defeq T$
and $T_2\defeq T(1+\eps)$, where
$\eps\defeq \exp(\tfrac{-C\log T}{\log\log T})$.
Then $\Delta\defeq\eps T$, $\bfT\defeq T+2\Delta$,
and
\[
\log\bfT\sim\log T
\mand
\sL^{-A}=\exp\big(\tfrac{(-2C+o(1))\log T}{\log\log T}\big)
\quad(T\to\infty).
\]
Clearly, $\bfT\sL^{-A}=o(\Delta\log\bfT)$
as $T\to\infty$. According to Lemma~\ref{lem:prime find},
there is a prime $x$ satisfying \eqref{eq:prime range} such
that $|x^{-iy}-1|>\frac{1}{\sqrt{2}}$. For any such prime~$x$,
the main term in \eqref{eq:super} dominates the error term,
and thus we have
\[
\ssum{\rho=\frac12+i\gamma\\T_1<\gamma<T_2}x^\rho\zeta(\rho+iy)\ne 0.
\]
The theorem follows.
\end{proof}

\subsection{Bounds on various sums}

\begin{lemma}\label{lem:series bound}
Let $y\in\R$, $x\ge 10$, and put
$c\defeq 1+\tfrac{1}{\log x}$. Then
\be\label{eq:S}
\sum_{n\ne x}\frac{|D_y(n)|}{n^c|\log(x/n)|}\ll \log^2x
\ee
\end{lemma}

\begin{proof}
The proof follows standard arguments.
We begin by observing that
\be\label{eq:Dyn upper bound}
\big|D_y(n)\big|=\bigg|-\sum_{ab=n}\Lambda(a)b^{-iy}\bigg|\le
\sum_{ab=n}\Lambda(a)=\log n\qquad(n\in\N).
\ee

Let $S$ denote the sum on the left side of \eqref{eq:S}.
For positive integers with $n\le\frac23x$ or $n\ge\frac43x$, 
we have $|\log(x/n)|\gg 1$; using \eqref{eq:Dyn upper bound}, 
the total contribution to $S$ from such integers is
\[
\ll\sum_n\frac{\log n}{n^c}
=-\zeta'(c)\asymp(c-1)^{-2}=\log^2x,
\]
which is acceptable. For integers $\frac23x<n<\frac43x$, we have
$D_y(n)\ll \log x$ by~\eqref{eq:Dyn upper bound}, and in this range we also have
$n^c\asymp n$. Thus, to finish the proof, it suffices to observe that both sums
\[
S_1\defeq\sum_{\frac23x<n<x}\frac{1}{n\log(x/n)}\mand
S_2\defeq\sum_{x<n<\frac43x}\frac{1}{n\log(n/x)}
\]
are of size $O(\log x)$, as one verifies using
standard arguments.
\end{proof}

\begin{lemma}\label{lem:summation bound}
Assume RH.
Let $y\ne 0$, $t>t'\ge 10$, and put $\Delta'\defeq t-t'$.
For any $\kappa$ in the range $\er\le \kappa\le t/\er$, we have
\[
\sum_{t'<n\le t}D_y(n)n^{iy}
\ll\Delta'\big\{|y|^{-1}+(\kappa/t)^{1/2}\log^2(t/\kappa+|y|)
+\log \kappa\big\}+t/\kappa
+(t\kappa)^{1/2}\log^2t.
\]
\end{lemma}

\begin{proof}
A well known result of von Koch \cite{vonKoch}
asserts that, under RH, one has
\be\label{eq:vonKoch}
\sum_{n\le x}\Lambda(n)=x+O\big(x^{1/2}(\log x)^2\big)\qquad(x\ge 2).
\ee
More generally, building on work of Gonek, Graham,
and Lee~\cite{GoGrLe}, the author and Sinha~\cite{BaSi}
have shown that the Riemann Hypothesis is true if and only if
the following uniform estimate holds:
\[
\sum_{n\le x}\Lambda(n)n^{iy}=\frac{x^{1+iy}}{1+iy}
+O\big(x^{1/2}\log^2(x+|y|)\big)\qquad(y\in\R,~x\ge 2).
\]
By partial summation, this implies
\be\label{eq:GGLps}
\sum_{n\le x}\frac{\Lambda(n)n^{iy}}{n}
\ll |y|^{-1}+x^{-1/2}\log^2(x+|y|)+1\qquad(y\ne 0,~x\ge 2)
\ee
(and the sum is zero for $x\in(0,2]$). Therefore, writing
\[
D_y(n)n^{iy}=-n^{iy}
\sum_{ab=n}\Lambda(a)b^{-iy}
=-\sum_{ab=n}\Lambda(a)a^{iy},
\]
we have
\[
\sum_{t'<n\le t}D_y(n)n^{iy}
=-\sum_{t'<ab\le t}\Lambda(a)a^{iy}
=-\ssum{t'<ab\le t\\a\le t/\kappa}\Lambda(a)a^{iy}
-\ssum{t'<ab\le t\\a>t/\kappa}\Lambda(a)a^{iy}.
\]
By \eqref{eq:GGLps} and the Chebyshev bound,
\dalign{
\ssum{t'<ab\le t\\a\le t/\kappa}\Lambda(a)a^{iy}
&=\sum_{a\le t/\kappa}\Lambda(a)a^{iy}
\bigg(\fl{\frac{t}{a}}-\fl{\frac{t'}{a}}\bigg)
=\Delta'\sum_{a\le t/\kappa}\frac{\Lambda(a)a^{iy}}{a}+O(t/\kappa)\\
&\ll\Delta'\big\{|y|^{-1}+(\kappa/t)^{1/2}
\log^2(t/\kappa+|y|)+1\big\}+t/\kappa.
}
On other hand, using \eqref{eq:vonKoch} we have
\dalign{
\ssum{t'<ab\le t\\a>t/\kappa}\Lambda(a)a^{iy}
&=\sum_{b\le \kappa}\ssum{t'/b<a\le t/b\\a>t/\kappa}\Lambda(a)a^{iy}
\ll\sum_{b\le \kappa}\sum_{t'/b<a\le t/b}\Lambda(a)\\
&=\sum_{b\le \kappa}
\bigg\{\frac{\Delta'}{b}+O\big((t/b)^{1/2}\log^2(t/b)\big)\bigg\}\\
&\ll\Delta'\log \kappa+(t\kappa)^{1/2}\log^2(t/\kappa).
}
Combining the two bounds, the lemma follows.
\end{proof}

\begin{lemma}\label{lem:technical bound}
Let $t\ge 10\,x\ge 100$, and put $c\defeq 1+\frac{1}{\log x}$. Then
\be\label{eq:puzzle}
\sum_{n\ge 2}\frac{\log n}{n^c(|t-2\pi nx|+t^{1/2})}
\ll\frac{(x+t^{1/2}\log t)\log t}{t^{c+1/2}}.
\ee
\end{lemma}

\begin{proof}
The sum on the left side of \eqref{eq:puzzle}
is equal to
\[
\bigg(\sum_{n\in S_1}+\sum_{n\in S_2}+\sum_{n\in S_3}\bigg)
\frac{\log n}{2\pi xn^c(|n-X|+Y)},
\]
where
\[
X\defeq\frac{t}{2\pi x},\qquad Y\defeq\frac{t^{1/2}}{2\pi x},
\]
and
\dalign{
S_1&\defeq\{n\ge 2:|n-X|\le Y\},\\
S_2&\defeq\{n\ge 2:Y<|n-X|\le 2X\},\\
S_3&\defeq\{n\ge 2:|n-X|>2X\}.
}
As $n\asymp X$ for each $n\in S_1$, and $|S_1|\ll Y+1$, the 
sum over $n\in S_1$ is
\[
\ll\frac{(\log X)(Y+1)}{xX^cY}
\ll\frac{(\log t)(x/t^{1/2}+1)}{t^c}.
\]
Since $n\ll X$ for each $n\in S_2$, the 
sum over $n\in S_2$ is
\[
\ll\frac{\log X}{xX^c}
\bigg(\sum_{2\le n<X-Y}+\sum_{X+Y<n\le 3X}\bigg)
\frac{1}{|n-X|+Y}
\ll\frac{\log X\log(X/Y)}{xX^c}\ll\frac{\log^2t}{t^c}.
\]
Finally, since $n>3X$ and $n-X>\frac12n$ for each $n\in S_3$,
the sum over $n\in S_3$ is
\[
\ll\frac{1}{x}\sum_{n>3X}\frac{\log n}{n^{c+1}}\ll\frac{\log X}{xX^c}
\ll\frac{\log t}{t^c}.
\]
Putting everything together, we obtain the stated bound.
\end{proof}

\subsection{The zeta function}

\begin{lemma}\label{lem:MV}
The estimate
\[
\frac{\zeta'}{\zeta}(s)=
\ssum{\rho=\beta+it\\|\gamma-t|<1}\frac{1}{s-\rho}
+O(\log|t|)
\]
holds uniformly for $-1\le\sigma\le 2$ and $|t|\ge 10$.
\end{lemma}

\begin{proof}
See, e.g., Montgomery and Vaughan~\cite[Lemma~12.1]{MontVau}.
\end{proof}

\begin{lemma}\label{lem:MV2}
Assume RH. There is an absolute constant $\lambda>0$ such that
\be\label{eq:lambdabd}
\big|\zeta(s)\big|\le
\exp\bigg(\frac{\lambda\log|t|}{\log\log|t|}\bigg)
\ee
holds uniformly for $\sigma\ge\frac12-\frac{1}{\log\log|t|}$
with $|t|\ge 10$.
\end{lemma}

\begin{proof}
By \cite[Theorem~13.18]{MontVau}, there is an absolute constant $\lambda_0>0$
such that
\[
\big|\zeta(s)\big|\le
\exp\bigg(\frac{\lambda_0}{\log\log|t|}\bigg)
\qquad(\sigma\ge\tfrac12,~|t|\ge 10).
\]
For $\sigma$ in the range $\frac12-\frac{1}{\log\log|t|}\le\sigma\le\frac12$,
we apply \cite[Corollary~10.5]{MontVau}:
\[
\big|\zeta(s)\big|\asymp
|t|^{1/2-\sigma}\big|\zeta(1-s)\big|.
\]
Since $\tfrac12-\sigma\le\tfrac{1}{\log\log|t|}$ and
$1-\sigma\ge\tfrac12$, we get that
\[
|t|^{1/2-\sigma}\le\exp\bigg(\frac{\log|t|}{\log\log|t|}\bigg)
\mand
\big|\zeta(1-s)\big|
\le\exp\bigg(\frac{\lambda_0\log|t|}{\log\log|t|}\bigg).
\]
The lemma is valid with $\lambda\defeq\lambda_0+1$.
\end{proof}

\subsection{Estimates with $\cX(s)$}

In \S\ref{sec:proof} below, we use the functional equations
\be\label{eq:fun1}
\zeta(s)=\zeta(1-s)\cX(s)
\ee
and
\be\label{eq:fun2}
\frac{\zeta'}{\zeta}(s)
=-\frac{\zeta'}{\zeta}(1-s)+\log\pi
-\tfrac12\psi\big(\tfrac12s\big)
-\tfrac12\psi\big(\tfrac12(1-s)\big),
\ee
where
\be\label{eq:barchetta}
\cX(s)\defeq 2^s\pi^{s-1}\Gamma(1-s)\sin\tfrac\pi 2s
\mand
\psi(s)\defeq\frac{\Gamma'}{\Gamma}(s).
\ee
Here, we collect some information about the behavior of $\cX$.

\begin{lemma}\label{lem:X expansion}
Let $\cI$ be a bounded interval in $\R$. Uniformly for $\sigma\in\cI$
and $t\ge 1$, we have
\be\label{eq:spray}
\cX(1-\sigma+it)=\er^{\pi i/4}
\exp\Big(\!-it\log\Big(\frac{|t|}{2\pi\er}\Big)\Big)
\Big(\frac{|t|}{2\pi}\Big)^{\sigma-1/2}\big\{1+O_\cI(|t|^{-1})\big\}.
\ee
\end{lemma}

\begin{proof}
From the definition of $\cX$ (see \eqref{eq:barchetta}) it follows that
\[
\cX(1-s)=2^{1-s}\pi^{-s}\Gamma(s)\sin\tfrac{\pi}{2}(1-s).
\]
Suppose $s=\sigma+it$ with $t\ge 1$.
Using Stirling's formula for the gamma function
\[
\Gamma(s)=\sqrt{2\pi}\,s^{s-1/2}\er^{-s}\{1+O(t^{-1})\}
\]
(see, e.g., Montgomery and Vaughan~\cite[Theorem~C.1]{MontVau})
along with the estimates
\[
(s-\tfrac12)\log s=(\sigma -\tfrac12)\log t+\sigma 
+(t\log t-\tfrac{\pi}{4})i+\tfrac{\pi is}{2}+O(t^{-1})
\]
and
\[
\sin\tfrac{\pi}{2}(1-s)
=\tfrac12\er^{-\pi i s/2}\{1+O(\er^{-\pi t})\},
\]
a straightforward computation leads to \eqref{eq:spray}.
\end{proof}

The next result is due to Gonek~\cite[Lemmas~2 and 3]{Gonek1};
the proof is based on the stationary phase method.

\begin{lemma}\label{lem:gonekLem8}
Uniformly for $10a\ge b>a\ge 10$, $\sigma\in[\frac{1}{10},10]$,
and $m\in\{0,1\}$,
\dalign{
&\int_a^b\exp\Big(\!-it\log\Big(\frac{t}{u\cdot\er}\Big)\Big)
\Big(\frac{t}{2\pi}\Big)^{\sigma-1/2}\Big(\log\frac{t}{2\pi}\Big)^m\,dt\\
&\qquad\qquad\qquad\qquad=(2\pi)^{1-\sigma}u^\sigma\er^{(u-\frac\pi 4)i}
\Big(\log\frac{u}{2\pi}\Big)^m\cdot\ind{}(a,b;u)
+O\big(E(\log a)^m\big),
}
where
\[
\ind{}(a,b;u)\defeq\begin{cases}
1&\quad\hbox{if $a<u\le b$},\\
0&\quad\hbox{otherwise},\\
\end{cases}
\]
and
\[
E\defeq a^{\sigma-1/2}+\frac{a^{\sigma+1/2}}{|a-u|+a^{1/2}}
+\frac{b^{\sigma+1/2}}{|b-u|+b^{1/2}}.
\]
\end{lemma}

The next lemma is a variant of 
Conrey, Ghosh, and Gonek \cite[Lemma~1]{ConGhoGon}.

\begin{lemma}\label{lem:CGG-get schwifty}
Define
\[
f_j(s)\defeq\begin{cases}
1&\quad\hbox{if $j=0$},\\
\psi(\frac{s}2)&\quad\hbox{if $j=+1$},\\
\psi(\frac{1-s}2)&\quad\hbox{if $j=-1$},\\
\end{cases}
\qquad
g_j(v)\defeq\begin{cases}
1&\quad\hbox{if $j=0$},\\
\log(\pi v)-\frac{\pi i}{2}&\quad\hbox{if $j=+1$},\\
\log(\pi v)+\frac{\pi i}{2}&\quad\hbox{if $j=-1$}.\\
\end{cases}
\]
For each $j\in\{0,\pm 1\}$, and uniformly for
$2T_1>T_2>T_1\ge 10(|y|+1)$, $c\in[\frac{1}{2},10]$, 
and $v>0$, we have the uniform estimate
\be
\label{eq:prunes}
\begin{split}
&\frac{1}{2\pi i}\int_{c-iT_2}^{c-iT_1}
\cX(1-s+iy)f_j(s)v^{-s}\,ds\\
&\qquad\qquad\qquad\qquad=v^{-iy}\e(v)g_j(v)\cdot\ind{}\big(T'_1,T'_2;2\pi v\big)
+O_y(E),
\end{split}
\ee
where $T'_j\defeq T_j+y$ for $j\in\{1,2\}$, and
\[
E\defeq \frac{(\log T'_1)^{|j|}}{v^c}\bigg\{(T'_1)^{c-1/2}
+\frac{(T'_1)^{c+1/2}}{|T'_1-2\pi v|+(T'_1)^{1/2}}
+\frac{(T'_2)^{c+1/2}}{|T'_2-2\pi v|+(T'_2)^{1/2}}\bigg\}.
\]
The implied constant depends only on $y$.
\end{lemma}

\begin{proof}
For $T_1,T_2$ in the stated range, it is immediate that
$10T'_1\ge T'_2>T'_1\ge 10$. Making the change of variables
$s\mapsto c-it+iy$, and applying Lemma~\ref{lem:X expansion},
the left side of \eqref{eq:prunes} is equal to
\[
\frac{\er^{\pi i/4}}{2\pi v^{c+iy}}\int_{T'_1}^{T'_2}
\exp\Big(\!-it\log\Big(\frac{t}{2\pi v\er}\Big)\Big)
\Big(\frac{t}{2\pi}\Big)^{c-1/2}f_j(c-it+iy)
\big\{1+O(t^{-1})\big\}\,dt.
\]
When $j=0$, both functions $f_j$ and $g_j$ are identically one, and
the estimate \eqref{eq:prunes} follows at once by applying
Lemma~\ref{lem:gonekLem8} with $u\defeq 2\pi v$ and $m\defeq 0$.
On the other hand, for $j=\pm 1$, we substitute the uniform bound
\[
f_j(c-it+iy)
=\log\frac{t}{2\pi}+\log\pi\mp\tfrac{\pi i}{2}+O_y(t^{-1})
\qquad(t\in[T_2,T_1]),
\]
into the preceding integral, and the estimate stated in the lemma
follows by applying Lemma~\ref{lem:gonekLem8} 
with $u\defeq 2\pi v$ and $m=0,1$.
\end{proof}

{\Large\section{Proof of Theorem~\ref{thm:main2}}
\label{sec:proof}}

\subsection{Preliminaries}

In what follows, $\Theta=\Theta(A)\ge 2$ is a constant that depends only
on~$A$; a specific value of $\Theta$ is given below; see \S\ref{sec:endgame}.
From now on, any implied constants in the symbols $O$, $\ll$, and $\gg$
may depend on $y$ and $A$ but are independent of all other parameters
unless indicated otherwise.

Let $T_0=T_0(y,A)$ be a large parameter satisfying
\begin{alignat}{2}
\label{eq:T0conds1}
T_0&\ge 10(|y|+1),\\
\label{eq:T0conds2}
\log\log T_0&\ge \Theta^2,\\
\label{eq:T0conds3}
\tfrac{\log T_0}{\log\log T_0}&\ge|y|^{-1}.
\end{alignat}
Let $T_1,T_2$ be real numbers such that $2T_1>T_2>T_1>T_0$, and denote
\[
T'_j\defeq T_j+y\qquad(j=0,1,2).
\]
This notation is used systematically in what follows.
If $T_0$ is large enough, then
\[
\bfT\asymp T_1\asymp T_2\asymp T'_1\asymp T'_2.
\]
Once and for all, let $x$ be an arbitrary \emph{prime number}
in the range
\be\label{eq:bigbound1}
\bfT\sL^{-\Theta}<x<\er^{\pi/|y|}\cdot\bfT\sL^{-\Theta}.
\ee
Then
\be\label{eq:loggylog}
\log x=\log\bfT-\frac{\Theta\log\bfT}{\log\log\bfT}+O(1)
=\log\bfT+O\bigg(\frac{\log\bfT}{(\log\log\bfT)^{1/2}}\bigg),
\ee
where we used \eqref{eq:T0conds2} in the last step.

In proving Theorem~\ref{thm:main2}, we can assume without loss
of generality that neither $T_1$ nor $T_2$ is the ordinate
of a zero of $\zeta(s)$. Denote
\[
b\defeq\tfrac12-\tfrac{1}{\log\log\bfT}
\mand 
c\defeq 1+\tfrac{1}{\log x},
\]
and note that
\be\label{eq:xcTc}
x^c=\er x\asymp x\asymp \bfT\sL^{-\Theta}
\mand
\bfT^c=\bfT\exp\big(\tfrac{\log\bfT}{\log x}\big)\asymp\bfT,
\ee
where we used \eqref{eq:loggylog} for the last estimate.

Finally, let $\euR_\bfT$ be the region in $\C$ defined by
\[
\euR_\bfT\defeq\big\{s\in\C:\sigma\ge b\text{~and~}
\tfrac{1}{100}\bfT\le |t|\le 100\bfT\big\}.
\]
Using Lemma~\ref{lem:MV2}, we deduce that
there is an absolute\footnote{The constant $\lambda$ can be
chosen independent of $y$; however, $\sL$ depends on $y$ in view of
\eqref{eq:T0conds1}.} 
constant $\lambda>0$ such that
\be\label{eq:bigbound2}
\big|\zeta(s\pm iy)\big|
\le\sL^\lambda\qquad\text{for all~}s\in\euR_\bfT.
\ee

\subsection{The residue theorem}
By Cauchy's theorem, we have
\dalign{
\ssum{\rho=\frac12+i\gamma\\T_1<\gamma<T_2}x^\rho\zeta(\rho+iy)
&=\frac{1}{2\pi i}\bigg(\int_{c+iT_1}^{c+iT_2}
+\int_{c+iT_2}^{b+iT_2}
+\int_{b+iT_2}^{b+iT_1}
+\int_{b+iT_1}^{c+iT_1}\bigg)
\euD_y(s)x^s\,ds\\
&=I_1+I_2+I_3+I_4\quad\text{(say)}.
}
We estimate each term separately.

In the half-plane $\{\sigma>1\}$, the series representation
$\euD_y(s)=\sum_n D_y(n)n^{-s}$
is valid, and therefore
\dalign{
I_1&=\frac{1}{2\pi}\int_{T_1}^{T_2}\euD_y(c+it,\chi)\,x^{c+it}\,dt
=\frac{x^c}{2\pi}\sum_n\frac{D_y(n)}{n^c}
\int_{T_1}^{T_2}\Big(\frac x{n}\Big)^{it}\,dt\\
&=\frac{\Delta}{2\pi}D_y(x)
+O\bigg(x^c\sum_{n\ne x}\frac{|D_y(n)|}{n^c|\log(x/n)|}\bigg).
}
Applying Lemma~\ref{lem:series bound} and recalling
\eqref{eq:Dyn lower bound}, we see that
\[
I_1=-\frac{\Delta}{2\pi}\log x+O(x\log^2x),
\]
where we used the fact that $x^c\asymp x$. Using \eqref{eq:loggylog},
and taking into account that $x\log^2x\ll\bfT\sL^{-\Theta+1}$, it follows
that
\be\label{eq:I1result}
I_1=-\frac{\Delta}{2\pi}\log\bfT+O\bigg(\frac{\Delta\log\bfT}{(\log\log\bfT)^{1/2}}
+\bfT\sL^{-\Theta+1}\bigg).
\ee

Next, using Lemma~\ref{lem:MV} it follows that
\be\label{eq:I2embrace}
I_2=\frac{1}{2\pi i}\ssum{\rho=\frac12+i\gamma\\|\gamma-T_2|<1}
\int_{c+iT_2}^{b+iT_2}
\frac{\zeta(s+iy)x^{s}\,ds}{s-\rho}
+O\bigg(M(\log \bfT)\int_{b}^c
x^{\sigma}\,d\sigma\bigg),
\ee
where
\[
M\defeq\max\big\{\big|\zeta(\sigma+iT'_2)\big|:
b\le\sigma\le c\big\}.
\]
Since $M\le\sL^\lambda$ by \eqref{eq:bigbound2}, 
$\log\bfT\le\sL$, and 
$\int_b^c x^\sigma\,d\sigma\ll x^c\asymp\bfT\sL^{-\Theta}$
by \eqref{eq:xcTc},
the error term in~\eqref{eq:I2embrace} is at most
$O(\bfT\sL^{-\Theta+\lambda+1})$.
To bound the main term, we follow
the method of Gonek~\cite[p.\,402]{Gonek2}.
Let $\rho=\frac12+i\gamma$ be a zero of $\zeta(s)$ included in
the sum in \eqref{eq:I2embrace}. By Cauchy's theorem, the integral
\[
J_\rho\defeq\int_{c+iT_2}^{b+iT_2}
\frac{\zeta(s+iy)x^s\,ds}{s-\rho}
\]
is equal to
\[
J_\rho=-2\pi i\,\zeta(\rho+iy)x^\rho
+\bigg(\int_{b+iT_2}^{b+i(T_2+10)}
+\int_{b+i(T_2+10)}^{c+i(T_2+10)}
+\int_{c+i(T_2+10)}^{c+iT_2}\bigg)
\frac{\zeta(s+iy)x^s\,ds}{s-\rho}.
\]
The residue is $O(\bfT^{1/2}\sL^{-\Theta/2+\lambda})$ by 
\eqref{eq:bigbound1} and \eqref{eq:bigbound2}. On the other hand, 
using the bound
\[
\frac{1}{|s-\rho|}\ll\begin{cases}
\log\log\bfT&\quad\hbox{if $\sigma=b$},\\
1&\quad\hbox{if $t=T_2+10$},\\
1&\quad\hbox{if $\sigma=c$},\\
\end{cases}
\]
and applying \eqref{eq:xcTc} and \eqref{eq:bigbound2},
the overall contribution from the three integrals above is
 $O(\bfT\sL^{-\Theta+\lambda}\log\log\bfT)$.
Hence, $J_\rho\ll\bfT\sL^{-\Theta+\lambda}\log\log\bfT$
for each zero $\rho$ in the sum \eqref{eq:I2embrace}.
Since there are at most $O(\log\bfT)$ such zeros, we deduce that
\be\label{eq:I2result}
I_2\ll\bfT\sL^{-\Theta+\lambda+1}.
\ee
By a similar argument,
\be\label{eq:I4result}
I_4\ll\bfT\sL^{-\Theta+\lambda+1}.
\ee

It remains to bound 
\[
I_3=\frac{1}{2\pi i}\int_{b+iT_2}^{b+iT_1}\frac{\zeta'}{\zeta}(s)\zeta(s+iy)x^s\,ds.
\]
Using the functional equations \eqref{eq:fun1}
and \eqref{eq:fun2}, followed by a change of variables
$s\mapsto 1-s$, we get that
\[
I_3=\frac{-1}{2\pi i}\int\limits_{b'-iT_2}^{b'-iT_1}
\bigg\{\hskip-3pt-\frac{\zeta'}{\zeta}(s)+\log\pi
-\tfrac12\psi\big(\tfrac{s}{2}\big)
-\tfrac12\psi\big(\tfrac{1-s}{2}\big)
\bigg\}\zeta(s-iy)\cX(1-s+iy)x^{1-s}\,ds,
\]
where
\[
b'\defeq 1-b=\tfrac12+\tfrac{1}{\log\log\bfT}.
\]
We shift the line of integration to the right, back to the vertical 
line $\{\sigma=c\}$. Since $b'>\frac12$, no poles are encountered
along the way, and so the residue is zero. Let $\Omega$ be
the set of points $s$ on the union of
the horizontal line segments $b'-iT_j\longrightarrow c-iT_j$ for $j=1,2$.
By \eqref{eq:bigbound2} we have
\[
\zeta(s-iy)\ll\sL^\lambda\qquad(s\in\Omega).
\]
Moreover, since the distance
from any zero of $\zeta(s)$ to the half-plane $\{\sigma\ge b'\}$ is at
least $\tfrac{1}{\log\log\bfT}$ (under RH), Lemma~\ref{lem:MV}
implies that
\[
\frac{\zeta'}{\zeta}(s)\ll\log\bfT\log\log\bfT\qquad(s\in\Omega).
\]
Also, by Lemma~\ref{lem:X expansion}
we have
\[
\cX(1-s+iy)\ll\bfT^{\sigma-1/2}\qquad(s\in\Omega).
\]
Next, it is well known that the estimate
\[
\psi(s)=\log s+O(\tau^{-1})
\]
holds uniformly in any half-plane $\{\sigma\ge\sigma_0>0\}$,
where $\tau=\tau(t)\defeq |t|+10$, where the implied
constant depends only on $\sigma_0$. Taking $\sigma_0\defeq\tfrac{1}{10}$
(say), this implies
\[
\psi\big(\tfrac{s}{2}\big)=O(\log\bfT)
\mand
\psi\big(\tfrac{1-s}{2}\big)=O(\log\bfT)\qquad(s\in\Omega).
\]
Finally, we have trivially
\[
x^{1-s}\ll \bfT^{1-\sigma}\qquad(s\in\Omega).
\]
Putting everything together, we conclude that
\be\label{eq:I3Js equation}
I_3=J_1+J_2+J_3+O(\bfT^{1/2}\sL^{\lambda+1}),
\ee
where
\dalign{
J_1&\defeq\frac x{2\pi i}\int_{c-iT_2}^{c-iT_1}
\frac{\zeta'}{\zeta}(s)\zeta(s-iy)\cX(1-s+iy)x^{-s}\,ds,\\
J_2&\defeq\frac{-x\log\pi}{2\pi i}\int_{c-iT_2}^{c-iT_1}
\zeta(s-iy)\cX(1-s+iy)x^{-s}\,ds,\\
J_3&\defeq\frac x{4\pi i}\int_{c-iT_2}^{c-iT_1}
\big\{\psi\big(\tfrac{s}{2}\big)+\psi\big(\tfrac{1-s}{2}\big)\big\}
\zeta(s-iy)\cX(1-s+iy)x^{-s}\,ds.
}

\subsection{Bound on $J_1$}
\label{sec:J1}

Everywhere along the line $\{\sigma=c\}$ we have
\[
\frac{\zeta'}{\zeta}(s)\zeta(s-iy)
=\sum_n\frac{\overline{D_y(n)}}{n^s};
\]
consequently,
\[
J_1=x\sum_n\overline{D_y(n)}\cdot\frac{1}{2\pi i}
\int_{c-iT_2}^{c-iT_1}
\cX(1-s+iy)(nx)^{-s}\,ds.
\]
Since $2T_1>T_2>T_1\ge 10(|y|+1)$ (see \eqref{eq:T0conds1}),
$v\defeq nx>0$, and $c\in[\frac{1}{10},2]$, we can apply
Lemma~\ref{lem:CGG-get schwifty} with $j\defeq 0$
estimate each integral in the above sum.
Recalling that $x$ is an \emph{integer}, and thus
$\e(v)=\e(nx)=1$ for each~$n$ in the sum, it follows that
$J_1=L_1+O(L_2)$ with
\dalign{
L_1&\defeq x^{1-iy}\sum_{T'_1/(2\pi x)<n\le T'_2/(2\pi x)}
\overline{D_y(n)}n^{-iy},\\
L_2&\defeq \sum_n \frac{\log n}{n^c}\bigg\{(T'_1)^{c-1/2}
+\frac{(T'_1)^{c+1/2}}{|T'_1-2\pi nx|+(T'_1)^{1/2}}
+\frac{(T'_2)^{c+1/2}}{|T'_2-2\pi nx|+(T'_2)^{1/2}}\bigg\},
}
where we have used the fact that $|D_y(n)|\le\log n$
(see \eqref{eq:Dyn upper bound}) and that $x^c\asymp x$.
Applying Lemma~\ref{lem:summation bound} with
\[
t'\defeq\frac{T'_1}{2\pi x}\asymp\sL^\Theta,\qquad
t\defeq\frac{T'_2}{2\pi x}\asymp\sL^\Theta,\qquad
\Delta'=t-t'=\frac{\Delta}{2\pi x},
\]
we see that $L_1$ is
\[
\ll\Delta\big\{|y|^{-1}
+(\kappa/\sL^\Theta)^{1/2}\log^2(\sL^\Theta/\kappa+|y|)
+\log \kappa\big\}+x\,\sL^\Theta/\kappa
+x\,(\sL^\Theta\kappa)^{1/2}\log^2\sL^\Theta
\]
uniformly for all $\kappa$ in
the range $\er\le\kappa\le T'_2/(2\pi\er x)$.
Let us recall that the main term of $I_1$ is
$-\frac{1}{2\pi}\Delta\log\bfT$; see \eqref{eq:I1result}.
To ensure that this term and other terms arising from $J_3$
(see~\S\ref{sec:J3} below) are not dominated by the error,
$\kappa$~cannot be too large;
specifically, we need $\log\kappa=o(\log\bfT)$.
We choose $\kappa\defeq\sL^{\Theta/4}$, and this leads to the bound
\[
L_1\ll\Delta\big\{|y|^{-1}
+\tfrac{\log\bfT}{\log\log\bfT}\big\}+x\sL^{3\Theta/4}
\ll\tfrac{\Delta\log\bfT}{\log\log\bfT}+\bfT\sL^{-\Theta/4},
\]
where we used \eqref{eq:T0conds3} in the last step.
On the other hand, using Lemma~\ref{lem:technical bound} to bound $L_2$,
taking into account the bounds $-\zeta'(c)\ll\log^2x$
and $x\ll\bfT\sL^{-\Theta}$, we see that
\[
L_2\ll\bfT^{c-1/2}\log^2x+\bfT^{c+1/2}\cdot
\frac{(x+\bfT^{1/2}\log\bfT)\log\bfT}{\bfT^{c+1/2}}
\ll\bfT\sL^{-\Theta+1}.
\]
Consequently,
\be\label{eq:J1bound}
J_1\ll\tfrac{\Delta\log\bfT}{\log\log\bfT}+\bfT\sL^{-\Theta/4}.
\ee

\subsection{Bound on $J_2$}
\label{sec:J2}

On the line $\{\sigma=c\}$ we have $\zeta(s-iy)=\sum_nn^{-s+iy}$;
thus,
\[
J_2=-x\log\pi\sum_n n^{iy}
\cdot\frac{1}{2\pi i}\int_{c-iT_2}^{c-iT_1}
\cX(1-s+iy)(nx)^{-s}\,ds
\]
As in \S\ref{sec:J1}, we use Lemma~\ref{lem:CGG-get schwifty} with $j=0$
to write $J_2=L_3+O(L_4)$ with
\dalign{
L_3&\defeq-x^{1-iy}\log\pi\cdot
\big|\{n:T'_1<2\pi nx\le T'_2\}\big|\\
L_4&\defeq \sum_n \frac{1}{n^c}\bigg\{(T'_1)^{c-1/2}
+\frac{(T'_1)^{c+1/2}}{|T'_1-2\pi nx|+(T'_1)^{1/2}}
+\frac{(T'_2)^{c+1/2}}{|T'_2-2\pi nx|+(T'_2)^{1/2}}\bigg\}.
}
Clearly,
\[
L_3=-x^{1-iy}\log\pi
\bigg\{\frac{\Delta}{2\pi x}+O(1)\bigg\}
\ll \Delta+\bfT\sL^{-\Theta}.
\]
Also, comparing the definitions of $L_2$ and $L_4$, it is clear that
\[
L_4\ll L_2\ll\bfT\sL^{-\Theta+1}.
\]
Consequently,
\be\label{eq:J2bound}
J_2\ll \Delta+\bfT\sL^{-\Theta+1}.
\ee

\subsection{Estimate for $J_3$}
\label{sec:J3}

Write $J_3=J'_3+J''_3$ with
\dalign{
J'_3&\defeq\frac x{4\pi i}\int_{c-iT_2}^{c-iT_1}
\psi\big(\tfrac{s}{2}\big)\zeta(s-iy)\cX(1-s+iy)x^{-s}\,ds,\\
J''_3&\defeq\frac x{4\pi i}\int_{c-iT_2}^{c-iT_1}
\psi\big(\tfrac{1-s}{2}\big)\zeta(s-iy)\cX(1-s+iy)x^{-s}\,ds.
}
Unfolding the integral $J'_3$ as in \S\ref{sec:J2} we have
\[
J'_3=\frac{x^{1-iy}}{2}\sum_n n^{iy}
\cdot\frac{1}{2\pi i}\int_{c-iT_2}^{c-iT_1}
\psi\big(\tfrac{s}{2}\big)\cX(1-s+iy)(nx)^{-s}\,ds,
\]
and then we apply Lemma~\ref{lem:CGG-get schwifty} with $j=+1$
to express $J'_3=L_5+O(L_6)$ with
\dalign{
L_5&\defeq\frac{x^{1-iy}}{2}
\sum_{T'_1/(2\pi x)<n\le T'_2/(2\pi x)}
\big(\log(\pi nx)-\tfrac{\pi i}{2}\big),\\
L_6&\defeq \sum_n \frac{\log T'_1}{n^c}\bigg\{(T'_1)^{c-1/2}
+\frac{(T'_1)^{c+1/2}}{|T'_1-2\pi nx|+(T'_1)^{1/2}}
+\frac{(T'_2)^{c+1/2}}{|T'_2-2\pi nx|+(T'_2)^{1/2}}\bigg\}.
}
Since $\log(\pi nx)-\tfrac{\pi i}{2}=\log\bfT+O(1)$
for all $n$ in the sum for $L_5$, we have
\[
L_5=\frac {x^{1-iy}}{2}\bigg(\frac{\Delta}{2\pi x}+O(1)\bigg)
\big(\log\bfT+O(1)\big)
=\frac{x^{-iy}}{4\pi}\cdot\Delta\log\bfT
+O(\Delta+\bfT\sL^{-\Theta+1}).
\]
We also have
\[
L_6\ll\log\bfT\cdot L_4\ll\bfT\sL^{-\Theta+2},
\]
and therefore
\[
J'_3=\frac{x^{-iy}}{4\pi}\cdot\Delta\log\bfT
+O(\Delta+\bfT\sL^{-\Theta+2}).
\]
As the same estimate holds for $J''_3$
(using Lemma~\ref{lem:CGG-get schwifty} with $j=-1$), we get that
\be\label{eq:J3estimate}
J_3=\frac{x^{-iy}}{2\pi}\cdot\Delta\log\bfT
+O(\Delta+\bfT\sL^{-\Theta+2}).
\ee

\subsection{Endgame}
\label{sec:endgame}

Combining the results
\eqref{eq:I1result} and
\eqref{eq:I2result}\,--\,\eqref{eq:J3estimate},
we derive the estimate
\[
\ssum{\rho=\frac12+i\gamma\\T_1<\gamma<T_2}x^\rho\zeta(\rho+iy)
=\tfrac{1}{2\pi}(x^{-iy}-1)\Delta\log\bfT
+O\bigg(\frac{\Delta\log\bfT}{(\log\log\bfT)^{1/2}}
+\bfT\sL^{-A}\bigg).
\]
provided that $\Theta\ge \max\{A+\lambda+2,4A\}$, and this
finishes the proof.

{\Large\section{Concluding remarks}}

Let us define
\[
N_{0,y}(T)\defeq\big|\big\{\rho=\tfrac12+i\gamma:\zeta(\rho)=0\text{~and~}
\zeta(\rho+iy)\ne 0\big\}\big|,
\]
where each zero is counted according to its multiplicity.
Theorem~\ref{thm:main} implies that for any fixed
$C'>0$ one has (under RH) the lower bound
\[
N_{0,y}(T)\gg\exp\bigg(\frac{C'\log T}{\log\log T}\bigg)
\]
for all large $T$. Almost certainly, $N_{0,y}(T)$ is much
larger than this. In fact, it is reasonable to expect that
\[
N_{0,y}(T)=N_0(T)+O_y(1)
\]
for all $y\ne 0$, where
\[
N_0(T)\defeq\big|\big\{\rho=\tfrac12+i\gamma:\zeta(\rho)=0\big\}\big|
=\frac{T}{2\pi}\log\frac{T}{2\pi\er}+O(\log T).
\]
In fact, since the total number of gaps between the zeros of $\zeta(s)$
on the line $\{\sigma=\tfrac12\}$ is countable, for every $y\ne 0$
one has $N_{0,y}(T)=N_0(T)$ for all $T>0$, with at most countably many
exceptions. Under the \emph{Linear Independence Conjecture}, for every
\emph{rational} $y\ne 0$ one has $N_{0,y}(T)=N_0(T)$ for all $T>0$.

\bigskip\section*{Acknowledgments}

The author thanks Ali Ebadi and the anonymous referee for useful
remarks and for pointing out flaws in the original version
of the manuscript.


\begin{thebibliography}{99}

\bibitem{BaSi}
W.~D.~Banks and S.~Sinha,
The Riemann Hypothesis via the generalized von Mangoldt function.
\emph{Funct.\ Approx.\ Comment.\ Math.} (to appear).
Online preprint: {\tt arXiv:2209.11768}

\bibitem{Borwein}
P.~Borwein, S.~Choi, B.~Rooney and A.~Weirathmueller (eds.),
\emph{The Riemann hypothesis. A resource for the afficionado and virtuoso alike.}
CMS Books in Mathematics, Springer, New York, 2008.

\bibitem{ConGhoGon}
J.~B.~Conrey, A.~Ghosh and S.~M.~Gonek,
Simple zeros of the Riemann zeta-function.
\emph{Proc.\ London Math.\ Soc.} (3) 76 (1998), no.~3, 497--522.

\bibitem{Gonek1}
S.~M.~Gonek,
Mean values of the Riemann zeta function and its derivatives.
\emph{Invent.\ Math.} 75 (1984), no.~1, 123--141.

\bibitem{Gonek2}
S.~M.~Gonek, An explicit formula of Landau and its applications to the theory of the zeta-function, in \emph{A tribute to Emil Grosswald: number theory and related
analysis}, 395--413, Contemp.\ Math., 143, Amer.\ Math.\ Soc., Providence, RI, 1993.

\bibitem{GoGrLe}
S.~M.~Gonek, S.~W.~Graham, and Y.~Lee,
The Lindel\"of hypothesis for primes is equivalent
to the Riemann hypothesis.
\emph{Proc.\ Amer.\ Math.\ Soc.} 148 (2020), no.~7, 2863--2875.

\bibitem{vonKoch}
H.~von Koch, 
Sur la distribution des nombres premiers.
\emph{Acta Math.} 24 (1901), no.~1, 159--182.

\bibitem{PairCorr1}
H.~L. Montgomery, The pair correlation of zeros of the zeta function, in {\it Analytic number theory (Proc. Sympos. Pure Math., Vol. XXIV, St. Louis Univ., St. Louis, Mo., 1972)}, 181--193, Proc. Sympos. Pure Math., Vol. XXIV, Amer. Math. Soc., Providence, RI.

\bibitem{MontVau}
H.~L.~Montgomery and R.~C.~Vaughan,
\emph{Multiplicative number theory. I. Classical theory.}
Cambridge Studies in Advanced Mathematics, 97.
Cambridge University Press, Cambridge, 2007.

\bibitem{Riemann}
B.~Riemann, Ueber die Anzahl der Primzahlen unter einer gegebenen Gr\"osse.
\emph{Monatsberichte der Berliner Akademie}, 1859.

\bibitem{Titchmarsh}
E.~C.~Titchmarsh, 
\emph{The theory of the Riemann zeta-function.}
Second edition. Edited and with a preface by D. R. Heath-Brown.
The Clarendon Press, Oxford University Press, New York, 1986.

\end{thebibliography}
\end{document}